\documentclass[a4paper,11pt,oneside]{amsart}
\usepackage[english]{babel}
\usepackage[T1]{fontenc}
\usepackage{amstext}
\usepackage{amsbsy}
\usepackage{amsopn}
\usepackage{amsfonts}
\usepackage{amsmath}
\usepackage{amssymb}
\usepackage{marginnote} 
 
\renewcommand{\phi}{\varphi}
\newcommand{\cc}{\mathcal{C}}
\newcommand{\e}{\varepsilon}
\newcommand{\op}{\mathop{\rm Op}\nolimits}
\newcommand{\la}{\lambda}
\swapnumbers
\newtheorem{definition}{Definition}
\newtheorem{theorem}[equation]{Theorem} 
\newtheorem{lemma}[equation]{Lemma}    
\newtheorem{corollary}[equation]{Corollary}

\theoremstyle{remark}

\numberwithin{equation}{section}
\newcommand{\ignore}[1]{{}}

\newcommand{\GG}{\mathbb{G}}
\newcommand{\HH}{\mathbb{H}}

\newcommand{\RR}{\mathbb{R}}
\renewcommand{\ss}{\mathcal{S}}

\newcommand{\8}{\infty}
\renewcommand{\d}{\partial}
\newcommand{\Hn}{{\HH^n}}
\newcommand{\Rn}{{\RR^n}}
\newcommand{\supp}{\mathrm{supp\,}}
\newcommand{\sgn}{\mathrm{sgn\,}}
 
\title[Flag kernels]
 {Invertibility in the Flag Kernels Algebra on the Heisenberg Group} 

\author{Grzegorz K\k{e}pa}
\address{Institute of Mathematics, University of Wroc{\l}aw,
pl. Grunwaldzki 2/4, 50-384 Wroc{\l}aw, Poland, email: grzegorz.kepa.wroclaw@gmail.com}
 
\begin{document}

\begin{abstract}
Flag kernels are tempered distributions which generalize these of Calderon-Zygmund type. For any homogeneous group $\mathbb{G}$ the class of operators which acts on $L^{2}(\mathbb{G})$ by convolution with a flag kernel is closed under composition. In the case of the Heisenberg group we prove the inverse-closed property for this algebra. It means that if an operator from this algebra is invertible on $L^{2}(\GG)$, then its inversion remains in the class.
\end{abstract}

\maketitle
\tableofcontents
\setcounter{tocdepth}{1} 
 

\section{Introduction}

A subalgebra $\mathcal{A}$ of the algebra $\mathcal{B}(\mathcal{H})$ of all bounded operators on a Hilbert space $\mathcal{H}$ is said to be \textit{inverse-closed} if every $a\in\mathcal{A}$ which is invertible in $\mathcal{B}(\mathcal{H})$ is also invertible in $\mathcal{A}$. The question whether an algebra of convolution operators on a Lie group, or simply $\Rn$ is \textit{inverse-closed} is not new. In 1953, Calder{\'o}n and Zygmund \cite{cz2} showed that the class of convolution operators on $L^{2}(\Rn)$ whose kernels are homogeneous of degree $-n$ and are locally in $L^{q}(\Rn)$ away from the origin, has the property. Here $1<q<\infty$.  Much later the result was generalized by Christ \cite{christ} who proved that similar algebras on a homogeneous group are \textit{inverse-closed}. A homogeneous group $\GG$ is a nilpotent Lie group with dilations, a very natural generalization of the homogeneous structure on $\Rn$.

Another direction has been taken by Christ and Geller \cite{christ-geller} who dealt with the algebra of operators with kernels which are homogeneous of degree $-n$ and smooth away from the identity on a homogeneous group with gradation. This algebra is \textit{inverse-closed} too. A step further has been made by G{\l}owacki \cite{87} who showed that this is so for any homogeneous group.

The kernels which are smooth away from the identity allow an interesting generalization. One can relax the homogeneity condition demanding only that the kernel satisfies the estimates
$$|\d_{x}^{\alpha}K(x)|\lesssim |x|^{-Q-|\alpha|},$$
where $Q$ is the homogeneous dimension of the group. The cancellation condition takes the form
$$|\langle K,\phi\rangle|=|\int_{\GG}\phi(x)K(x)dx|\lesssim \Vert\phi\Vert,$$
for $\phi\in\ss (\GG)$, where $\Vert\cdot\Vert$ is a fixed seminorm in the Schwartz space $\ss (\GG)$. Such kernels $K$ are often called the \textit{Calder{\'o}n-Zygmund} kernels and the corresponding operators $\op (K)$ the \textit{Calder{\'o}n-Zygmund} operators. The class is closed under the composition of operators (Cor{\'e}-Geller \cite{core}) so they form an algebra. It has been proved recently (G{\l}owacki \cite{cz3}) that this algebra is \textit{inverse-closed} as well.

Let us specify the notion to the Heisenberg group which is the group under study in this paper. For the sake of simplicity, let us consider here only the one-dimensional case.
The underlying manifold of $\HH$ is $\RR^{3}$ which we write down as
$$\HH=\HH_{1}\oplus\HH_{2}=\RR^{2}\oplus\RR.$$
In these coordinates the group law is
$$(w,t)\circ (v,s)=(w+v,t+s+w_{1}v_{2}),$$
where $w=(w_{1},w_{2})$, $v=(v_{1},v_{2})\in\RR^{2}$, $t,s\in\RR$. There are many choices of compatible dilations, but the most natural is
\[
\delta_{r}(w,t)=(rw,r^{2}t),\qquad r>0.
\]   
In this setting the size condition for a \textit{Calder{\'o}n-Zygmund} kernel reads
$$|\partial_{w}^{\alpha}\d_{t}^{\beta}K(w,t)|\lesssim (|w|+|t|)^{-4-\alpha-2\beta}.$$
The cancellation condition doesn't get any simpler, so we do not repeat it here.

We are going to compare these conditions with the estimates that define \textit{flag kernels} which are the main object of study in this paper. Flag kernels were introduced by M{\"u}ller-Ricci-Stein \cite{muller} and Nagel-Ricci-Stein \cite{nagel} in their study of Marcinkiewicz multipliers (the first paper) and CR manifolds (the other one). These kernels are much more singular than the \textit{Calder{\'o}n-Zygmund} kernels. Accordingly, the definition is more complex. We consider a tempered distribution $K$ on $\HH$ which is smooth for $w\neq 0$ and satisfies the estimates
$$|\d_{w}^{\alpha}\d_{t}^{\beta}K(w,t)|\lesssim |w|^{-2-\alpha}(|w|+|t|)^{-2-2\beta},\qquad w\neq 0,$$
as well as the following three cancellation conditions:

1) For every $\phi\in\ss (\HH_{1})$, the distribution
$$f\mapsto\int_{\HH}\phi(w)f(t)dwdt$$
is a \textit{Calder{\'o}n-Zygmund} kernel on $\HH_{2}$,

2) For every $\phi\in\ss (\HH_{2})$, the distribution
$$f\mapsto\int_{\HH}f(w)\phi(t)dwdt$$
is a \textit{Calder{\'o}n-Zygmund} kernel on $\HH_{1}$,

3) For every $\phi\in\ss (\HH)$,
$$|\int_{\HH}\phi(w,t)dwdt|\lesssim 1.$$
Finally, for given $\alpha,\beta$, the estimates are uniform with respect to $\phi$ if $\phi$ stays in a bounded set in the respective Schwartz space.

The operators with flag kernels share some properties with the \textit{Calder{\'o}n-Zygmund} operators. They are bounded on $L^{p}(\GG)$-spaces and form an algebra (see M{\"u}ller-Ricci-Stein \cite{muller}, Nagel-Ricci-Stein \cite{nagel}, Nagel-Ricci-Stein-Wainger \cite{nagel-ricci}, G{\l}owacki \cite{colloquium2010}, G{\l}owacki \cite{lp}). We are, however, interested in the inversion problem for this class. Before going any further, let us pause for a moment and consider the simplest case, namely that of an Abelian group $\Rn$. Then, the Fourier transform $\widehat{K}$ is a function on $\Rn$ which is smooth away from the origin and satisfies the estimates
\begin{align*}
|\d^{\alpha}_{\xi}\widehat{K}(\xi)|\lesssim |\xi|^{-|\alpha|},\qquad\xi\neq 0.
\end{align*}
These estimates are, equivalent to the ones defining the \textit{Calder{\'o}n-Zygmund} kernel. If the operator $\op (K)$ is invertible, then
$$|\widehat{K}(\xi)|\geqslant c>0,\ \ \ 0\neq\xi\in\Rn,$$
and it is directly checked that $\widehat{L}=1/\widehat{K}$ satisfies analogous estimates, so that $L$ is a flag kernel such that $L\star K=K\star L=\delta_{0}$.

A similar idea works for the Heisenberg group $\HH$. Let $\pi^{\la}$ denote the Schr{\"o}dinger representation of $\HH$ with the Planck constant $\la\neq 0$, If $K$ is a flag kernel on $\HH$ such that the operator $\op (K)$ is invertible, then, for every $\la\neq 0$, the operator $\pi_{K}^{\la}$ is invertible and can be regarded as a pseudodifferential operator in a suitable class. By the Beals theorem, the inverse belongs to the same class. Now, the estimates are uniform in $\la$, so one can recover the kernel of the inverse operator from the kernels of $(\pi_{K}^{\la})^{-1}$ and show that it is a flag kernel. Thus the algebra of the operators with flag kernels on the Heisenberg group turns out to be \textit{inverse-closed}. We believe that similar method could be used in the case of a general 2-step nilpotent Lie group.

There is a technicality in the proof we want to comment on. There exists no universal definition  of the extension of the unitary representation to a space of distributions. One has to rely on specific properties of the distribution space in question. Everything works fine for distributions with compact support. A \textit{Calder{\'o}n-Zygmund} kernel can be split into a compactly supported part and a part which is square integrable, so there is no problem with the definition of $\pi_{K}^{\la}$. No such splitting is available for flag kernels. Instead we modify the domain of the distribution. Originally, a distribution is a functional on the Schwartz space. We introduce two other spaces on which flag kernels can be regarded as continuous functionals. The cancellation conditions are important here. We also take adventage of a functional calculus of G{\l}owacki \cite{arkiv2007}. Once $\pi_{K}^{\la}$ is defined for our flag kernel, we can follow the path outlined above.
  

\section{Preliminaries}

The main structure of this work is the Heisenberg group. As a set it is 
\[\HH^{n}=\Rn\times\Rn\times\RR.
\]
Elements of the group will usually be denoted by
\[
\HH^{n}\ni h=(x,y,t)=(v,t).
\]
The group multiplication is 
\[
(x,y,t)\cdot(x',y',t')=(x+x',y+y',t+t'+xy').
\]
We define a homogeneous norm on $\Hn$ as
\[
|h|=\Vert v\Vert+|t|^{\frac{1}{2}}=\sum_{i=1}^{2n}|v_{i}|+|t|^{\frac{1}{2}}
\]
with the corresponding family of dilations
\[
\delta_{j}(h)=\delta_{j}(v,t)=(jv,j^{2}t),\qquad j>0,
\]
in the sense that $|\delta_{j}(h)|=j|h|$. 
The set $\{\delta_{j}\}_{j>0}$ actually forms a group of automorphisms. The homogeneous dimension is the number $Q=2n+2$. We will use the designations
\[
\d^{\gamma}_{h}=\d_{v}^{\alpha}\d_{t}^{\beta}=\d_{v_{1}}^{\alpha_{1}}\d_{v_{2}}^{\alpha_{2}}...\d_{v_{2n}}^{\alpha_{2n}}\d_{t}^{\beta}
\]
and
\[
|\gamma|=|\alpha|+2\beta=\sum_{i=1}^{2n}\alpha_{i}+2\beta,
\]
where $\alpha=(\alpha_{1},\alpha_{2},\dots,\alpha_{2n}),\ \alpha_{k},\beta\in\mathbb{N}$.
One of the main tools is the abelian Fourier transform defined by the formula
$$\widehat{f}(\zeta):=\int_{\Hn}f(h)e^{-2\pi ih\zeta}dh,$$
where
\[
\HH_{n}\ni \zeta=(\xi,\eta,\lambda)=(w,\lambda),
\hspace{2cm}
\HH_{n}=\RR_{n}\times\RR_{n}\times\RR.
\]
It can be first defined for Schwartz functions
\[
\mathcal{S}(\HH^{n})=\{f\in\mathcal{C}^{\infty}(\HH^{n}):\forall N\in\mathbb{N}\sup_{h\in\Hn}\max_{|\gamma|\leqslant N}|\partial_{h}^{\gamma}f(h)|(1+|h|)^{N}<\infty\}
\]
and then lifted to the Lebesque space of the square-integrable functions
\[
L^{2}(\Hn)=\{f:\int_{\Hn}|f(h)|^{2}dh<\infty\}
\]
or to the space of tempered distributions $\ss'(\Hn)$ wchich is the space of all continuous linear functionals on $\ss(\Hn)$ in the sense of the usual seminorm topology.
For $S\in\ss '(\Hn),\ f\in\ss (\Hn)$ one can put
$$\langle\widehat{S},f\rangle:=\langle S,\widehat{f}\rangle.$$
Let also
\[
f^{\star}(x)=\overline{f(x^{-1})},\quad\langle S^{\star},f\rangle:=\langle S,f^{\star}\rangle,\quad S\in\ss'(\Hn),\ f\in\ss (\Hn).
\]
By $\delta_{0}$ we will denote the Dirac distribution.


\section{Flag kernels and their convolution operators}

Automorphisms $\{\delta_{j}\}_{j>0}$ decompose our group $\Hn$ into their eigenspaces
$$\GG_{1}\oplus\GG_{2}\ni (v,t).$$
Theorem 2.3.9 of Nagel-Ricci-Stein \cite{nagel} says that there is a one-to-one correspodence between flag kernels and their multipliers. It allows us to define flag kernels as follows.
\begin{definition}
Let $\Hn$ be the Heisenberg group and
$$\mathbb{H}_{n}=\GG_{1}^{\star}\oplus\GG_{2}^{\star}\ni (w,\la)$$ the dual vector space to $\Hn$. We say that a tempered distribution $K$ is a flag kernel iff its Fourier transform $\widehat{K}$ agrees with a smooth function outside of hyperspace
$\{(w,\la):\la=0\}$ and satisfies the estimates
\begin{center}
$|\partial_{w}^{\alpha}\partial_{\lambda}^{\beta}\widehat{K}(w,\lambda)|\leqslant c_{\alpha,\beta}(\Vert w\Vert+|\lambda|^{1/2})^{-|\alpha|}|\lambda|^{-\beta}$,\ \ \ all $\alpha,\beta$.
\end{center}
\end{definition}
Observe that in particular $\widehat{K}$ belongs to $L^{\infty}(\HH_{n})$.
For $f\in\mathcal{S}(\Hn)$, $K\in\mathcal{S}'(\Hn)$ we define their convolution as $$K\star f(h):=\langle K,\, _{h^{-1}}\!\widetilde{f}\rangle=\int_{\Hn}K(h')f(h'^{-1}h)dh',$$
where $\widetilde{f}(h)=f(h^{-1})$.
Our point of departure are the following two theorems:
\begin{theorem}
Let $K$ be a flag kernel on the Heisenberg group. Then
\[
\Vert K\star f\Vert_{2}\lesssim \Vert f\Vert_{2},\qquad f\in\ss(\Hn).
\]
Hence we have an $L^{2}$-bounded operator $\op(K)f:=K\star f$.
\end{theorem}
\begin{theorem}\label{splot}
Let $K,S$ be flag kernels on the Heisenberg group $\Hn$ and $$T:=\op(K)\op(S).$$ Then, there exists a flag kernel $L$ such that $T=\op(L)$.
\end{theorem}
Thus the flag kernels give rise to convolution operators, bounded on $L^{2}(\Hn)$, which form a subalgebra of $\mathcal{B}(L^{2}(\Hn))$. For convenience we will write $L=K\star S$.

These are theorems of Nagel-Ricci-Stein \cite{nagel} who proved them for a class of homogeneous groups which includes all two-step homogeneous groups (Theorems 2.6.B and 2.7.2). Partial results can be found in an earlier paper of M{\"u}ller-Ricci-Stein \cite{muller} (Theorem 3.1). They were subsequently generalized for all homogeneous group independently and virtually simultanously by Nagel-Ricci-Stein-Wainger \cite{nagel-ricci} and G{\l}owacki \cite{colloquium2010}.
The aim of this paper is to prove the following theorem.
\begin{theorem}\label{nasze}
Let $\Hn$ be the Heisenberg group. Let $K$ be a flag kernel on $\Hn$. Suppose that the operator $\op(K)$ is invertible on $L^{2}(\Hn)$. Then there exists a flag kernel $L$, such that for all $f\in L^{2}(\Hn)$
$$\op(K)^{-1}f=L\star f=\op(L)f.$$
\end{theorem}
Observe that $\op(K)$ is translation invariant.
Further the same holds for its inversion.
By general theory it follows that there exists a tempered distribution $L$ such that $\op(K)^{-1}f=L\star f$. Now it suffices to show that $L$ is a flag kernel. In the following considerations we can assume that the flag kernel $K$ is symmetric, i.e. $K=K^{\star}$. In fact if Theorem \ref{nasze} is true for such kernels, let us pick an arbitrary flag kernel $K$. Then, we can consider kernels $K^{\star}\star K$ and $K\star K^{\star}$ which are symmetric. By Theorem \ref{splot} they are both flag kernels. Therefore, by Theorem \ref{nasze}, there exist flag kernels $S,T$ such that
\[
S\star(K^{\star}\star K)=\delta_{0}
\hspace{1cm}
\&
\hspace{1cm}
(K\star K^{\star})\star T=\delta_{0}.
\]
Again, by Theorem \ref{splot} and associativity, it follows that there exist flag kernels $L_{1},L_{2}$ such that
\[
L_{1}\star K=\delta_{0}
\hspace{1cm}
\&
\hspace{1cm}
K\star L_{2}=\delta_{0}.
\]
The identity
$$L_{1}=L_{1}\star (K\star L_{2})=(L_{1}\star K)\star L_{2}=L_{2}$$
ends the proof of our theorem for an arbitrary flag kernel $K$.    
 

\section{Schr{\"o}dinger representation}

\begin{definition}
For $\lambda\neq 0$ and $h\in\Hn$ we define the family of operators $$\{\pi_{h}^{\lambda}:h\in\Hn,\lambda\neq 0\},$$ all acting on the same $L^{2}(\Rn)$ by the following formula
\[
\pi_{h}^{\la}f(s):=\left\{
\begin{array}{ll}
e^{2\pi i\la t}e^{2\pi i\sqrt{\la}ys}f(s+\sqrt{\la}x);\ \la>0,
\cr
e^{2\pi i\la t}e^{2\pi i\sqrt{|\la|}ys}f(s-\sqrt{|\la|}x);\ \la<0.
\end{array}
\right.
\]
\end{definition}
For a Hilbert space $\mathcal{H}$, denote by $\mathcal{U}(\mathcal{H})$, $\mathcal{B}(\mathcal{H})$ the spaces of all unitary and bounded operators, respectively. It is well-known (see, e.g. Folland \cite{folland}, sec. 1.3) that, for every $\la\neq 0$,
$$\Hn\ni h\longmapsto\pi_{h}^{\la}\in\mathcal{U}(L^{2}(\Rn))$$
is a unitary representation on the Hilbert space $L^{2}(\Rn)$.
$$L^{1}(\Hn)\ni f\longmapsto\pi_{f}^{\la}\in\mathcal{B}(L^{2}(\Rn))$$
is a representation of $\star$-algebra $L^{1}(\Hn)$ on the Hilbert space $L^{2}(\Rn)$. 
 
 
\section{Useful notation}

For $f,g\in\ss(\Rn)$ we define the function
\[
c_{f,g}(x,y):=\int_{\Rn}e^{2\pi iyu}f(u+x)g(u)du.
\]
In particular
\[
\widehat{c_{f,g}}(\xi,\eta)=\widehat{f}(\xi)g(\eta)e^{2\pi i\xi\eta}
.\]
Let also
\[
C_{f,g}^{\lambda}(x,y,t):=\langle\pi_{h}^{\la}f,g\rangle
\]
Let $\la>0$. One can calculate that
\begin{align*}
C_{f,g}^{\lambda}(x,y,t)=\int_{\Rn}\pi_{(x,y,t)}^{\la}f(u)g(u)du=e^{2\pi it\la}c_{f,g}(\sqrt{\la}x,\sqrt{\la}y).
\end{align*}
We also have
\begin{align*}
\widehat{c_{f,g}\circ\delta_{\sqrt{\la}}}(\xi,\eta)&=\int\int\int e^{-2\pi ix\xi}e^{-2\pi iy\eta}e^{2\pi i\sqrt{\la}yu}f(u+\sqrt{\la}x)g(u)dudxdy
\\
&=\int\int\int\la^{-n/2}e^{-2\pi ix\frac{\xi}{\sqrt{\la}}}e^{2\pi iu\frac{\xi}{\sqrt{\la}}}e^{-2\pi iy\eta}e^{2\pi i\sqrt{\la}yu}f(x)g(u)dudxdy\\
&=\int\int\la^{-n/2}\widehat{f}(\frac{\xi}{\sqrt{\la}})e^{2\pi iu(\frac{\xi}{\sqrt{\la}}+\sqrt{\la}y)}e^{-2\pi iy\eta}g(u)dudy\\
&=\int\la^{-n/2}\widehat{f}(\frac{\xi}{\sqrt{\la}})g^{\vee}(\frac{\xi}{\sqrt{\la}}+\sqrt{\la}y)e^{-2\pi iy\eta}dy\\
&=\int|\la|^{-n}\widehat{f}(\frac{\xi}{\sqrt{\la}})g^{\vee}(y)e^{-2\pi iy\frac{\eta}{\sqrt{\la}}}e^{2\pi i\frac{\xi}{\sqrt{\la}}\frac{\eta}{\sqrt{\la}}}dy\\
&=|\la|^{-n}\widehat{f}(\frac{\xi}{\sqrt{\la}})g(\frac{\eta}{\sqrt{\la}})e^{2\pi i\frac{\xi}{\sqrt{\la}}\frac{\eta}{\sqrt{\la}}}=|\la|^{-n}\widehat{c_{f,g}}(\frac{\xi}{\sqrt{\la}},\frac{\eta}{\sqrt{\la}}).
\end{align*}
Moreover
\[
\widehat{C_{f,g}^{\la}}(\xi,\eta,r)=\widehat{c_{f,g}\circ\delta_{\sqrt{\la}}}\otimes\widehat{e^{2\pi i(\cdot)\la}}(\xi,\eta,r)=|\la|^{-n}\widehat{c_{f,g}}\otimes\delta_{\la}(\frac{\xi}{\sqrt{\la}},\frac{\eta}{\sqrt{\la}},r),
\]
where $\delta_{\la}$ is a Dirac distribution supported at $\la$. For $\la<0$ the above formula should be slightly modified. For such $\la$ one can get analogously
\[
\widehat{C_{f,g}^{\la}}(\xi,\eta,r)=-|\la|^{-n}\widehat{c_{f,g}}\otimes\delta_{\la}(-\frac{\xi}{|\la|^{1/2}},\frac{\eta}{|\la|^{1/2}},r).
\]

Suppose that $a$ is a function on $\Rn\times\Rn$ which is bounded or square-integrable. Then, the weakly defined operator
\begin{align*}
\langle Af,g\rangle &=\int\int e^{2\pi i\xi\eta}a(\xi,\eta)\widehat{f}(\xi)g(\eta)d\eta d\xi 
\\
&=\int\int a(\xi,\eta)\widehat{c_{f,g}}(\xi,\eta)d\xi d\eta=\langle a,\widehat{c_{f,g}}\rangle
\end{align*} 
is a continuous mapping from $\ss (\Rn)$ to $\ss' (\Rn)$. It is often denoted by $A=a(x,D)$ and called a pseudodifferential operator with the KN (Kohn-Nirenberg) symbol $a$. 


\section{The class $\ss_{0}$ and the operator $\pi_{K}^{\la}$}

Let $g$ be a function on $\HH_{n}$ such that $g(w,\la)=\phi(\la)$, where $\phi\in\mathcal{C}_{c}^{\infty}(\mathbb{G}_{2}^{\star})$. Then
$g^{\vee}(u,t)=\delta_{0}\otimes\phi^{\vee}(u,t)$.
Observe that if for example $f\in\ss(\Hn)$, then
\begin{align*}
f\star g^{\vee}(h)&=\int f(hr^{-1})dg^{\vee} (r)=\int f(h-r)dg^{\vee} (r)
\\
&=\int f(r^{-1}h)dg^{\vee} (r)=g^{\vee}\star f(h),
\end{align*}
so $g^{\vee}$ is a central measure.
We will need a notion of the $\la$-support of a function $f$. By definition a real number $\la_{0}$ is not in the $\la$-$\supp(f)$ iff there exists $\varepsilon$ such that no point $(w,\la)$, where $\la\in (\la_{0}-\varepsilon,\la_{0}+\varepsilon)$, is in the support of $f$.
\begin{definition}\label{lambda}
We say that a Schwartz function $f$ is in $\ss_{0} (\Hn)$ iff
\[
(\exists\e>0)\forall\phi\in\cc_{c}^{\infty}((-\e,\e)\cup(-1/\e,-\infty)\cup(1/\e,\infty))\,\widehat{f}\phi=0,
\]
that is, iff $\la$-$\supp (f)$ is bounded and does not contain $0$.
\end{definition}
\begin{lemma}
Suppose $f\in\ss_{0} (\Hn)$ and $K$ is a flag kernel. Then $K\star f$ is in $\ss_{0} (\Hn)$.
\end{lemma}
\begin{proof}
Let us define first $a\# b:=(a^{\vee}\star b^{\vee})^{\wedge}$. As $f$ is in $\ss_{0}(\Hn)$ take $\varepsilon,\phi$ which satisfy  the condition of the definition \ref{lambda}. We have
\begin{align*}
\widehat{K\star f}\phi=(K\star f\star\phi^{\vee})^{\wedge}=\widehat{K}\#\widehat{f}\phi=0,
\end{align*}
so the same $\varepsilon$ works also for $K\star f$.
It remains to explain why $K\star f$ is an element of $\ss(\Hn)$.
Let $\psi\in\cc_{c}^{\infty}(\RR\setminus\{0\})$ be equal to 1 on $\la$-$\supp$ of $f$. Observe that $\psi^{\vee}$ can be thought of as a central measure. Thus
\begin{align*}
K\star f=K\star f\star\psi^{\vee}=K\star\psi^{\vee}\star f=K_{1}\star f,
\end{align*}
where $\widehat{K_{1}}$ is smooth and
\begin{align*}
|\d_{w}^{\alpha}\d_{\la}^{\beta}\widehat{K_{1}}(w,\la)|\leqslant c_{\alpha,\beta}(1+\Vert w\Vert+|\la|^{1/2})^{-|\alpha|}(1+|\la|)^{-\beta}.
\end{align*}
Now if we write that $a\in Sym^{N,M}(\Hn)$ iff
\begin{align*}
|\d_{w}^{\alpha}\d_{\la}^{\beta}\widehat{a}(w,\la)|\leqslant c_{\alpha,\beta}(1+\Vert w\Vert+|\la|^{1/2})^{-|\alpha|-N}(1+|\la|)^{-\beta-M},
\end{align*}
then $K_{1}\in Sym^{0,0}(\Hn)$, $f\in Sym^{N,M}(\Hn)$ for all $N,M$ because it is a Schwartz function. By G{\l}owacki's symbolic calculus \cite{arkiv2007} (Theorem 6.4) governed by the metric
\[
g_{(w,\la)}(u,r)=\frac{\Vert u\Vert}{1+\Vert w\Vert+|\la|^{1/2}}+\frac{|r|}{(1+|\la|^{1/2})^{2}};\qquad (w,\la)\in\mathbb{H}_{n},\ (u,r)\in\mathbb{H}_{n}
\]
we have
\begin{align*}
K\star f&=K_{1}\star f\in Sym^{0,0}(\Hn)\star \bigcap_{N,M}Sym^{N,M}(\Hn)\\
&\subseteq\bigcap_{N,M}Sym^{0,0}(\Hn)\star Sym^{N,M}(\Hn)
\\
&\subseteq\bigcap_{N,M}Sym^{N,M}(\Hn)\cong \ss (\Hn).
\end{align*}
\end{proof}
\begin{lemma}
The class $S_{0} (\Hn)$ is dense in $L^{2}(\Hn).$
\end{lemma}
\begin{proof}
Let $g\in L^{2}(\Hn)$ be such that $\forall f\in S_{0}(\Hn)\ \langle g,f\rangle=0.$
Then $\langle\hat{g},\hat{f}\rangle=0$. Hence $\supp\hat{g}\subseteq\RR^{2n}\times\{0\}.$
But it implies that $g=0$ almost everywhere.
\end{proof}
\begin{lemma}
The G{\"a}rding space $$\mathcal{G}^{\la}:=\{\pi_{\phi}^{\la}f: \phi\in S_{0}(\Hn), f\in L^{2}(\Rn)\}$$
is dense in $L^{2}(\Rn).$
\end{lemma}
\begin{proof}
Take any $g\in L^{2}(\Rn)$ such that for all $\phi\in S_{0},\ f\in L^{2}(\Rn)$ we have
$\langle g,\pi_{\phi}^{\la}f\rangle =0.$ We will show that $g=0$ a.e.
Consider only those functions $\phi$ which can be decomposed as $\phi(x,y,t)=\phi_{1}(x,y)\phi_{2}(t).$ Then
\begin{align*}
0&=\langle g,\pi_{\phi}^{\la}f\rangle =\langle g,\int_{\Hn}\phi(h)\pi_{h}^{\la}fdh\rangle=\int_{\Hn}\phi(h)\langle g,\pi_{h}^{\la}f\rangle dh\\
&=\int_{\RR}\left(\int_{\RR^{2n}}\phi_{1}(x,y)\langle g,\pi_{h}^{\la}f\rangle dxdy\right)\phi_{2}(t)dt=\int_{\RR}F(t)\phi_{2}(t)dt.
\end{align*}
It follows that $\langle\hat{F},\hat{\phi_{2}}\rangle=0$, which, by the structure of $\ss_{0} (\Hn)$, implies that 
$$\supp\hat{F}\subseteq\{0\}.$$
Hence $\hat{F}=\sum_{n=0}^{N}c_{n}\delta_{0}^{(n)}.$
By the Schwarz inequality
$||F||_{\8}\leq ||f||_{2}||g||_{2}||\phi_{1}||_{1}$. So $F\in L^{\8}(\RR)$ and at the same time
$F(t)=\sum_{n=0}^{N}c_{n}t^{n}.$
Consequently, there is no other option than $F=const$, which means that
$$\int_{\RR^{2n}}\phi_{1}(x,y)\langle g,\pi_{h}^{\la}f\rangle dxdy=c_{\phi_{2}}.$$
By the density of $S(\RR^{2n})$ in $L^{2}(\RR^{2n})$, we have that the expression $\langle g,\pi_{h}^{\la}f\rangle$ does not depend on the variable $t$. 
So $(\pi_{(0,0,t)}^{\la}-I)g=0$ for all $t$,
which leads to a contradiction unless $g=0$ a.e.
\end{proof}
\begin{lemma}
Let $K,L$ be flag kernels such that $K\upharpoonright_{\ss_{0}(\Hn)}=L\upharpoonright_{\ss_{0}(\Hn)}$. Then, $K=L$.
\end{lemma}
\begin{proof}
If $\langle K,f\rangle =\langle L,f\rangle$ for $f\in\ss_{0}(\Hn)$, then
$\langle K-L,f\rangle=0$ and so
$\langle\widehat{K-L},\widehat{f}\rangle =0$. Therefore from the definition of the class $\ss_{0}(\Hn)$ for every $\la$ nonzero element of the center of $\mathbb{H}_{n}$, we have $\widehat{K-L}=0.$ Thus $\widehat{K}=\widehat{L}$ as elements of $L^{\infty}(\mathbb{H}_{n})$, so $K=L$ in $\ss'(\Hn)$.
\end{proof}
\begin{definition}
We denote by $\mathcal{B}_{0}(\Hn)$ the class of all smooth functions such that their Fourier transforms are bounded measures whose $\la$-support does not contain $0$. One can norm this space with $\Vert f\Vert_{\mathcal{B}_{0}}=\Vert\widehat{f}\Vert_{\mathcal{M}}$, where $\Vert\cdot\Vert_{\mathcal{M}}$ denotes the total variation of a measure.
\end{definition}
Observe that $\ss_{0}(\Hn)\subset\mathcal{B}_{0}(\Hn)$. Moreover it also contains objects of type $C_{f,g}^{\la}$, as $\Vert C_{f,g}^{\la}\Vert_{\mathcal{B}_{0}}\leqslant\Vert\widehat{c_{f,g}}\Vert_{1}$. Now as $\ss_{0}(\Hn)$ is total for flag kernels we can extend such a kernel from $\ss_{0}(\Hn)$ to $\mathcal{B}_{0}(\Hn)$ by the formula
\[
\langle K,f\rangle=\int_{\mathbb{H}_{n}}\widehat{K}(w,\la)d\widehat{f}(w,\la).
\]
Continuity is gained for free as $|\langle K,f\rangle|\leqslant\Vert \widehat{K}\Vert_{\infty}\Vert\widehat{f}\Vert_{\mathcal{M}}$.
Now we can define the representation of a flag kernel. Suppose first that $K\in\ss_{0}(\Hn)$. Then,
\begin{align*}
\langle\pi_{K}^{\la}f,g\rangle =\langle\int_{\Hn} K(h)\pi_{h}^{\la}fdh,g\rangle=\int_{\Hn} K(h)\langle\pi_{h}^{\la}f,g\rangle dh=\langle K,C^{\la}_{f,g}\rangle,
\end{align*}
for $f,g\in\ss (\Rn)$.
Hence, for every flag kernel we put $\langle\pi_{K}^{\la}f,g\rangle:=\langle K,C^{\la}_{f,g}\rangle$ as a weak definition of its representation. Observe next that if $\la>0$
\begin{align*}
\langle\pi_{K}^{\la}f,g\rangle &=\int\int K(u,t)c_{f,g}(|\la|^{1/2}u)e^{2\pi it\la}dudt\\
&=\int\int\int\widetilde{\widehat{K}}(\xi,\eta,r)|\la|^{-n}\widehat{c_{f,g}}(\frac{\xi}{|\la|^{1/2}},\frac{\eta}{|\la|^{1/2}})d\xi d\eta d\delta_{\la}(r)\\
&=\int\int\widetilde{\widehat{K}}(|\la|^{1/2}\xi,|\la|^{1/2}\eta,\la)\widehat{c_{f,g}}(\xi,\eta)d\xi d\eta.
\end{align*}
Similar calculation for $\la<0$ leads to a conclusion that $\pi_{K}^{\la}$ is a pseudodifferential operator with the KN symbol
$$\widetilde{\widehat{K}}(\sgn(\la)|\la|^{1/2}\xi,|\la|^{1/2}\eta,\la).$$
\begin{lemma}
Let $K$ be a flag kernel and $\phi$ in $\ss_{0} (\Hn)$. Then, the operators $\pi_{K\star\phi}^{\la}$ and $\pi_{K}^{\la}\pi_{\phi}^{\la}$ are equal.  
\end{lemma}
\begin{proof}
First one can calculate that
\begin{align*}
C_{\pi_{\phi}^{\la}f,g}^{\la}(h)&=\int\pi_{h}^{\la}\pi_{\phi}^{\la}f(s)g(s)ds=\int\pi_{h}^{\la}\int\phi(h')\pi_{h'}^{\la}f(s)dh'g(s)ds\\
&=\int\phi(h')\int\pi_{hh'}^{\la}f(s)g(s)dsdh'=\int\widetilde{\phi}(h')C_{f,g}^{\la}(hh'^{-1})dh'\\
&=C_{f,g}^{\la}\star\widetilde{\phi}(h).
\end{align*}
Now using fact that $K\star\phi$ is in $\ss_{0} (\Hn)$ 
\begin{align*}
\langle\pi_{K\star\phi}^{\la}f,g\rangle& =\int K\star\phi(h)\langle\pi_{h}^{\la}f,g\rangle dh=\int K\star\phi(h)C_{f,g}^{\la}(h)dh\\
&=\int K(h)C_{f,g}^{\la}\star\widetilde{\phi}(h)dh=\int K(h)C_{\pi_{\phi}^{\la}f,g}^{\la}(h)dh=\langle\pi_{K}^{\la}\pi_{\phi}^{\la}f,g\rangle.
\end{align*}
\end{proof}
\begin{corollary}
Let $K_{1},K_{2}$ be flag kernels. The operators $\pi_{K_{1}\star K_{2}}^{\la}$ and $\pi_{K_{1}}^{\la}\pi_{K_{2}}^{\la}$ are equal.
\end{corollary}
\begin{proof}
As flag kernels form an algebra, the above lemma implies
\begin{align}\label{splatanie}
\pi_{K_{1}\star K_{2}}^{\la}\pi_{\phi}^{\la}f=\pi_{K_{1}\star K_{2}\star\phi}^{\la}f=\pi_{K_{1}}^{\la}\pi_{K_{2}\star\phi}^{\la}f=\pi_{K_{1}}^{\la}\pi_{K_{2}}^{\la}\pi_{\phi}^{\la}f,
\end{align}
where the second equality follows by the fact that $K_{2}\star\phi\in\ss_{0} (\Hn)$. \ref{splatanie} proves that the operators agree on vectors of type $\pi_{\phi}^{\la}f$ which are dense in $L^{2}(\Rn)$ when $\phi\in\ss_{0} (\Hn)$, $f\in L^{2}(\Rn)$.
\end{proof} 
 

\section{Representations of $L^{2}$}

We start from a simple calculation of the Kohn-Nirenberg symbol of $\pi_{f}^{\la}$, where $f$ is a Schwartz function. Let $\la$ be positive. We have
\begin{align*}
\pi_{f}^{\la}u(s)&=\int f(x,y,t)e^{2\pi it\la}e^{2\pi i\sqrt{\la}ys}u(s+\sqrt{\la}x)dxdydt\\
&=\int f(x,\sqrt{\la}s^{\vee},\la^{\vee})u(s+\sqrt{\la}x)dx
=|\la|^{-n/2}\int f(\frac{x-s}{\sqrt{\la}},\sqrt{\la}s^{\vee},\la^{\vee})u(x)dx\\
&=|\la|^{-n/2}\int\int f(\frac{x-s}{\sqrt{\la}},\sqrt{\la}s^{\vee},\la^{\vee})e^{2\pi ix\xi}\widehat{u}(\xi)d\xi dx\\
&=\int\int f(x,\sqrt{\la}s^{\vee},\la^{\vee})e^{2\pi i\sqrt{\la}x\xi}e^{2\pi is\xi}\widehat{u}(\xi)d\xi dx\\
&=\int f^{\vee}(\sqrt{\la}\xi,\sqrt{\la}s,\la)e^{2\pi is\xi}\widehat{u}(\xi)d\xi.
\end{align*}
As we can see in this case the symbol of $\pi_{f}^{\la}$ is also $$a(\xi,\eta)=\widetilde{\widehat{f}}(\sgn(\la)|\la|^{1/2}\xi,|\la|^{1/2}\eta,\la).$$
Let for a moment $x,y\in\Hn$ and $A$ be a Hilbert-Schmidt operator on $\ss(\Hn)$ with a kernel $\Omega$. One can calculate that
\begin{align*}
Au(x)=\int \Omega(x,y)u(y)dy=\int\Omega(x,y)\int e^{2\pi iy\xi}\widehat{u}(\xi)d\xi dy=\int\Omega(x,\xi^{\vee})\widehat{u}(\xi)d\xi.
\end{align*}
It is easy to see that if $a$ is the KN symbol of $A$, then
$$a(\xi,\eta)=e^{-2\pi i\xi\eta}\Omega(\xi,\eta^{\vee}),$$
and by the Plancharel formula,
\begin{align*}
\Vert A\Vert_{HS}=\Vert\Omega\Vert_{2}=\int\int|\Omega(\xi,y)|^{2}d\xi dy=\int\int|e^{-2\pi i\xi\eta}\Omega(\xi,\eta^{\vee})|^{2}d\xi d\eta=\Vert a\Vert_{2}.
\end{align*}
The sign change on the first coordinate, in a situation where $\la$ is negative, have no impact on the obstacles with which we struggle. Thus from now on in all calculations we will disregard this difference.  
\begin{lemma}\label{l2}
Let $f\in L^{2}(\Hn)$ and $\{f_{n}\}_{n}\subset\ss_{0} (\Hn)$ be such that $f_{n}\rightarrow f$ in $L^{2}$. For almost every $\la$, there exists a subsequence $\{f_{n_{k}(\la)}\}_{k}$ such that $\pi^{\la}_{f_{n_{k}(\la)}}$ tend to an operator $A^{\la}$ in the Hilbert-Schmidt norm. $A^{\la}$ depends neither on the chosen sequence $f_{n}$ nor on its subsequence $f_{n_{k}(\la)}$. Moreover the KN symbol of $A^{\la}$ is $a_{\la}(w)=\widetilde{\widehat{f}}(|\la|^{1/2}w,\la)$.
\end{lemma}
\begin{proof}
By Plancherel's formula
\begin{align*}
\Vert f_{n}-f\Vert_{2}^{2}&=\int |\widehat{f_{n}}(w,\la)-\widehat{f}(w,\la)|^{2}dwd\la\\
&=\int|\la|^{n}\int|\widetilde{\widehat{f_{n}}}(|\la|^{1/2}w,\la)-\widetilde{\widehat{f}}(|\la|^{1/2}w,\la)|^{2}dwd\la\\
&=\int|\la|^{n}\Vert\pi_{f_{n}}^{\la}-A^{\la}\Vert_{HS}^{2}d\la,
\end{align*}
where $A^{\la}$ is the Hilbert-Schmidt operator with the symbol $(w,\la)\mapsto\widetilde{\widehat{f}}(|\la|^{1/2}w,\la)$.
By Fatou's lemma $\liminf\Vert\pi_{f_{n}}^{\la}-A^{\la}\Vert_{HS}=0$ for almost every $\la$.
Therefore, for almost every $\la$, there exists a subsequence $f_{n_{k}(\la)}$ such that $$\Vert\pi_{f_{n_{k}(\la)}}^{\la}-A^{\la}\Vert_{HS}\rightarrow 0.$$
\end{proof}
Let $\widehat{f}^{\la}(w):=\widehat{f}(|\la|^{1/2}w,\la)$.
Lemma \ref{l2} says that for every $u,v\in\ss(\Hn)$, $\langle\widehat{f}^{\la}_{n},\widehat{c_{u,v}}\rangle$ tends to $\langle\widehat{f}^{\la},\widehat{c_{u,v}}\rangle$ a.e. which implies that $\langle\pi_{f_{n}}^{\la}u,v\rangle$ must have a limit. This limit is $A^{\la}$ an it will be denoted by $\pi_{f}^{\la}$. Nevertheless, for a given $L^{2}$ function, the operator exists only for a.e. $\la$.
\begin{lemma}
Let $K$ be a flag kernel on $\Hn$. Then, for every $f\in L^{2}(\Hn)$,
$$\pi_{K}^{\la}\pi_{f}^{\la}=\pi_{K\star f}^{\la}.$$
\end{lemma}
\begin{proof}
Let $\{f_{n}\}_{n}\subset\ss_{0} (\Hn)$ tend to $f$ in $L^{2}(\Hn)$. By definition
$$\pi_{K\star f}^{\la}=\lim\pi_{K\star f_{n_{k}(\la)}}^{\la},$$ for a.e. $\la$. As $f_{n_{k}(\la)}$ converges to $f$ in $L^{2}(\Hn)$, by Lemma \ref{l2}, $K\star f_{n_{k}(\la)}$ converges to $K\star f$ a.e. Hence
\begin{align*}
\pi_{K}^{\la}\pi_{f}^{\la}=\lim\pi_{K}^{\la}\pi_{f_{n_{k_{s}}(\la)}}^{\la}=\lim\pi_{K\star f_{n_{k_{s}}(\la)}}^{\la}=\pi_{K\star f}^{\la}.
\end{align*}
\end{proof}
Assume that $f\in L^{2}(\Hn)$. Let us continue with the calculation of kernel $\Omega_{f}^{\la}$ of the operator $\pi_{f}^{\la}$. As it has been said before, we have
$$\Omega_{f}^{\la}(\xi,\eta^{\vee})=e^{2\pi i\xi\eta}f^{\vee}(|\la|^{1/2}\xi,|\la|^{1/2}\eta,\la).$$
Therefore,
\begin{align*}
\Omega_{f}^{\la}(\xi,y)=\int e^{-2\pi i\eta(y-\xi)}f^{\vee}(|\la|^{1/2}\xi,|\la|^{1/2}\eta,\la)d\eta=|\la|^{-n/2}f(|\la|^{1/2}\xi^{\vee},\frac{y-\xi}{|\la|^{1/2}},\la^{\vee}).
\end{align*}
Furthermore, by Plancherel's formula
\begin{align*}
\mathfrak{G}_{f}(\la):&=|\la|^{n}\Vert\pi_{f}^{\la}\Vert_{HS}^{2}=|\la|^{n}\Vert\Omega_{f}^{\la}\Vert_{2}^{2}=\int\int|f(|\la|^{1/2}\xi^{\vee},\frac{y-\xi}{|\la|^{1/2}},\la^{\vee})|^{2}d\xi dy\\
&=\int\int|f(x,y,\la^{\vee})|^{2}dxdy.
\end{align*}
The function $\mathfrak{G}_{f}$ is continuous when $f\in\ss (\Hn)$.
Let $A$ be linear, bounded operator on $L^{2}(\Hn)$ (in particular a convolver), $\chi_{E}$ a characteristic function of a set $E\subset\RR$.
Then, from the above calculation we can conclude that $$\chi_{E}(\la)\mathfrak{G}_{Af}(\la)=\mathfrak{G}_{A(\chi_{E}(\la)f)}(\la).$$
Let as recall here that $\Vert f\Vert_{2}^{2}=\int_{\RR^{\star}}\mathfrak{G}_{f}(\la)d\la$.
\begin{lemma}\label{calki}
Let A be a linear, bounded operator on $L^{2}(\Hn)$. Suppose that, for every $f\in L^{2}(\Hn)$ $$\int_{\RR^{\star}}\mathfrak{G}_{Af}(\la)d\la\geqslant c^{2}\int_{\RR^{\star}}\mathfrak{G}_{f}(\la)d\la,$$
then, for almost every $\la$, $\mathfrak{G}_{Af}(\la)\geqslant c^{2}\mathfrak{G}_{f}(\la).$ 
\end{lemma}
\begin{proof}
Assume a contrario that for a function $g$ and $\la$ in a set $E$ of positive Lebesque measure we have that $\mathfrak{G}_{Ag}(\la)<c^{2}\mathfrak{G}_{g}(\la)$. Then, there exists $\varepsilon$>0 and a subset $F$ of $E$ of positive Lebesque measure such that
$\mathfrak{G}_{Ag}(\la)\leqslant (1-\varepsilon)c^{2}\mathfrak{G}_{g}(\la)$ on $F$. Therefore,
\begin{align*}
c^{2}\int_{\RR^{\star}}\mathfrak{G}_{\chi_{F}(\la)g}(\la)d\la&\leqslant\int_{\RR^{\star}}\mathfrak{G}_{A(\chi_{F}(\la)g)}(\la)d\la=\int_{F}\mathfrak{G}_{Ag}(\la)d\la\\
&\leqslant c^{2}(1-\varepsilon)\int_{F}\mathfrak{G}_{g}(\la)d\la
=c^{2}(1-\varepsilon)\int_{\RR^{\star}}\mathfrak{G}_{\chi_{F}(\la)g}(\la)d\la,
\end{align*}
which is obviously a contradiction. 
\end{proof}
The same holds true for the opposite inequality and the proof is analogous.
\begin{theorem}
Let $K$ be a symmetric flag kernel, such that $\op(K)$ is invertible. The family $\{\pi_{K}^{\la}\}_{\la}$ is uniformly invertible, that is all $\pi_{K}^{\la}$ are invertible and the family of operators $\{(\pi_{K}^{\la})^{-1}\}_{\la}$ is uniformly bounded on $L^{2}(\Rn)$.
\end{theorem}
\begin{proof}
As $\op(K)$ is invertible there exists a constant $C_{K}$, such that for $f\in L^{2}(\Hn)$ $\Vert K\star f\Vert_{2}\geqslant C_{K}\Vert f\Vert_{2}$.
Using Plancherel formula we have
\begin{align*}
\int_{\RR^{\star}}|\la|^{n}\Vert\pi_{K\star f}^{\la}\Vert_{HS}^{2}d\la=\Vert K\star f\Vert_{2}^{2}\geqslant C_{K}^{2}\Vert f\Vert_{2}^{2}=C_{K}^{2}\int_{\RR^{\star}}|\la|^{n}\Vert\pi_{f}^{\la}\Vert_{HS}^{2}d\la.
\end{align*}
Now by Lemma \eqref{calki}
\begin{align*}
C_{K}\Vert\pi_{f}^{\la}\Vert_{HS}\leqslant\Vert\pi_{K\star f}^{\la}\Vert_{HS}=\Vert\pi_{K}^{\la}\pi_{f}^{\la}\Vert_{HS}\leqslant\Vert\pi_{K}^{\la}\Vert_{2\rightarrow 2}\Vert\pi_{f}^{\la}\Vert_{HS}.
\end{align*}
Consider the operator $\mathcal{P}_{g,h}$; $g,h\in L^{2}(\Rn)$, where $\Vert h\Vert_{2}\neq 0$, which acts on vectors 
$u\in L^{2}(\Rn)$ by $\mathcal{P}_{g,h}(u):=\langle u,g\rangle h$. It is easy to see that the kernel of $\mathcal{P}_{g,h}$ is $\Omega_{\mathcal{P}}(x,y)=g(x)h(y)$, so $\Vert\mathcal{P}_{g,h}\Vert_{HS}=\Vert g\Vert_{2}\Vert h\Vert_{2}$. Now
\begin{align*}
\pi_{K}^{\la}\mathcal{P}_{g,h}u=\langle \pi_{K}^{\la}u,g\rangle h=\langle u,(\pi_{K}^{\la})^{\star}g\rangle h=\langle u,\pi_{K^{\star}}^{\la}g\rangle h=\langle u,\pi_{K}^{\la}g\rangle h=\mathcal{P}_{\pi_{K}^{\la}g,h}u.
\end{align*} 
Hence
\begin{align*}
\Vert\pi_{K}^{\la}g\Vert_{2}\Vert h\Vert_{2}=\Vert\pi_{K}^{\la}\mathcal{P}_{g,h}\Vert_{HS}\geqslant C_{K}\Vert\mathcal{P}_{g,h}\Vert_{HS}=C_{K}\Vert g\Vert_{2}\Vert h\Vert_{2}.
\end{align*}
Dividing both sides by $\Vert h\Vert_{2}$ we obtain that
$\Vert \pi_{K}^{\la}g\Vert_{2}\geqslant C_{K}\Vert g\Vert_{2}$ holds for every $g\in\ L^{2}(\Rn)$ which, together with the fact that $\pi_{K}^{\la}$ is self-adjoint, implies our claim.
\end{proof}
 

\section{The Beals theorem and the main result}

Summing up our previous results we conclude that, for every $\la$, a flag kernel $K$, gives rise to an operator $\pi_{K}^{\la}$ which acts on $L^{2}(\Rn)$ as a pseudodifferential operator with the Kohn-Nirenberg symbol
$a_{\la}(w)=\widetilde{\widehat{K}}(|\la|^{1/2}w,\la)$. Moreover,
\begin{align*}
|\d_{w}^{\alpha}a_{\la}(w)|&=|\partial_{w}^{\alpha}\{\widetilde{\widehat{K}}(|\la|^{1/2}w,\la)\}|\leqslant c_{\alpha}(\Vert|\la|^{1/2}w\Vert+|\la|^{1/2})^{-|\alpha|}|\la|^{|\alpha|/2}\\
&=c_{\alpha}(1+\Vert w\Vert)^{-|\alpha|}.
\end{align*}
Observe that these estimates do not depend on $\la$. In particular, by the Calder{\'o}n-Vaillancourt theorem, the family $\{\pi_{K}^{\la}\}_{\la}$ is uniformly bounded on $L^{2}(\Rn)$.
Let us define
$$Sym^{0}(\RR^{2n}):=\{a\in\mathcal{C}^{\infty}(\RR^{2n}):|\partial_{w}^{\alpha}a(w)|\leqslant c_{\alpha}(1+\Vert w\Vert)^{-|\alpha|}\}.$$
In this language the family of symbols $\{a_{\la}\}_{\la}$ is bounded in $Sym^{0}(\RR^{2n})$ with the natural seminorm topology.
The key point in our argument is the following application of a much more general theorem of Beals.
\begin{theorem}[Beals \cite{beals}, Thm. 4.7]
Let $A=a(x,D)$, where $a\in Sym^{0}(\RR^{2n})$, be invertible on $L^{2}(\Rn)$. Then, $A^{-1}=b(x,D)$ with $b\in Sym^{0}(\RR^{2n})$. Each seminorm of $b$ depends only on a finite number of seminorms of $a$ and the operator norm of $A^{-1}$.
\end{theorem}

Let us denote the symbol of $(\pi_{K}^{\la})^{-1}$ by $b_{\la}$.
One can conclude that the seminorms of $b_{\la}$ once again do not depend on $\la$. So the family $\{b_{\la}\}_{\la}$ also corresponds to a bounded family in $Sym^{0}(\RR^{2n})$. Following G{\l}owacki \cite{arkiv2007}, we say that $a$ is a weak limit of a bounded sequence $\{a_{n}\}_{n}$ in $Sym^{0}(\RR^{2n})$ iff for every $\alpha$ the sequence $\{\d^{\alpha}a_{n}\}_{n}$ converges almost uniformly to $\d^{\alpha}a$. The twisted multiplication $\#$ is continuous in the weak sense. In the language of symbols the equation $\pi_{K}^{\la}(\pi_{K}^{\la})^{-1}=Id$ corresponds to $a_{\la}\#b_{\la}=1$. 
\begin{lemma}
The family $\{b_{\la}\}_{\la}$ is weakly smooth in the parameter $\la$. 
\end{lemma} 
\begin{proof}
We proceed by induction. Let $\{\la_{n}\}_{n}$ converge to a nonzero $\la$. As $\{b_{\la_{n}}\}_{n}$ is bounded in $Sym^{0}(\RR^{2n})$, we can use Arzeli-Ascoli theorem to find a weakly convergent subsequence. Let $\{b_{\la_{n_{k}}}\}_{k}$ tend to $b_{\la}(\{n_{k}\})$. We have $$1=b_{\la_{n_{k}}}\#a_{\la_{n_{k}}}\rightarrow b_{\la(\{n_{k}\})}\#a_{\la}.$$ Hence for every convergent subsequence $\{b_{\la_{n_{k}}}\}_{k}$, the limit must be the same and equal to $b_{\la}$. Therefore, it also must be the limit of $\{b_{\la_{n}}\}_{n}$.
Assume now that $\d_{\la}^{N}b_{\la}$ is continuous for $N<M$.
Observe that using continuity of $b_{\la}$ which we have just obtained, formally we have
\begin{align}\label{rozklad}
\lim_{h\rightarrow 0}\frac{b_{\la+h}-b_{\la}}{h}=\lim_{h\rightarrow 0}b_{\la}\#\frac{a_{\la}-a_{\la+h}}{h}\#b_{\la+h}=-b_{\la}\#\d_{\la}a_{\la}\#b_{\la},
\end{align}
where the right hand side is weakly continuous. Consider the set $$\Xi:=\{M=(M_{1},M_{2},M_{3}):M_{2}>0,M_{1}+M_{2}+M_{3}=M\}.$$ Iterating the decomposition (\ref{rozklad}) we obtain 
\begin{align}\label{pelny}
\d_{\la}^{M}b_{\la}=\sum_{M\in\Xi}c_{M}\d_{\la}^{M_{1}}b_{\la}\#\d_{\la}^{M_{2}}a_{\la}\#\d_{\la}^{M_{3}}b_{\la}.
\end{align}
By induction hypothesis the right hand side is again weakly continuous. Hence the proof is complete.
\end{proof}
The decomposition (\ref{pelny}) actually gives more. It turns out that $b(w,\la):=b_{\la}(w)$ is also smooth if only $\la\neq 0$. It is a consequence of the fact that every derivative of $b(w,\la)$ has bounded partial derivatives outside of every set of type $\RR^{2n}\times [-\varepsilon,\varepsilon]$.

Note that $\widetilde{\widehat{B}}(w,\la)=\widehat{B}(-w,-\la)$, so $\widetilde{\widehat{B}}$ is the Fourier transform of a flag kernel if and only if $\widehat{B}$ is.

\begin{theorem}
Let $B$ be a distribution such that $\widehat{B}(w,\la)=b_{\la}(|\la|^{-1/2}w)$. Then, $B$ is a flag kernel.
\end{theorem}
\begin{proof}
It is obvious from the definition that $\widehat{B}$ is smooth away from the hyperspace $\{(w,\la):\la=0\}$. We have
\begin{align*}
|(\d_{1,2,...,2n}^{\alpha}\widehat{B})(|\la|^{1/2}w,\la)||\la|^{|\alpha|/2}&=|\d_{1,2,...,2n}^{\alpha}\{\widehat{B}(|\la|^{1/2}w,\la)\}|=|\d_{w}^{\alpha}b_{\la}(w)|\\
\leqslant c_{\alpha}(1+\Vert w\Vert)^{-|\alpha|}.
\end{align*}
Therefore,
\begin{align*}
|(\d_{1,2,...,2n}^{\alpha}\widehat{B})(|\la|^{1/2}w,\la)|\leqslant c_{\alpha}(1+\Vert w\Vert)^{-|\alpha|}|\la|^{-|\alpha|/2}=c_{\alpha}(|\la|^{1/2}+\Vert|\la|^{1/2}w\Vert)^{-|\alpha|}.
\end{align*}
Now putting $w$ instead of $|\la|^{1/2}w$ we obtain
\begin{align}\label{indukcja}
|(\d_{w}^{\alpha}\widehat{B})(w,\la)|\leqslant c_{\alpha}(|\la|^{1/2}+\Vert w\Vert)^{-|\alpha|}.
\end{align}
It sufficies now to get the estimates of the derivatives with respect to $\la$. We can treat inequality (\ref{indukcja}) as an initial step of an induction.
First of all using the fact that $K$ is a flag kernel one can calculate that
\begin{align*}
|\d_{\la}^{M}a_{\la}(w)|&=|\d_{\la}^{M}\{\widetilde{\widehat{K}}(|\la|^{1/2}w,\la)\}|\\
&=\mid\sum_{1\leqslant |\beta|+j\leqslant M}\frac{c_{\beta,j}w^{\beta}}{|\la|^{M-j-|\beta|/2}}(\d_{1,2,...,2n}^{\beta}\d^{j}_{2n+1}\widetilde{\widehat{K})}(|\la|^{1/2}w,\la)\mid\\
&\leqslant\sum_{1\leqslant |\beta|+j\leqslant M}\frac{|c_{\beta,j}|\Vert w\Vert^{|\beta|}}{|\la|^{M-j-|\beta|/2}}(\Vert |\la|^{1/2}w\Vert+|\la|^{1/2})^{-|\beta|}|\la|^{-j}\\
&\lesssim\frac{\Vert w\Vert^{|\beta|}}{|\la|^{M}(1+\Vert w\Vert)^{|\beta|}}\leqslant|\la|^{-M}.
\end{align*}
As $\{b_{\la}\}_{\la}$ is a bounded family in $Sym^{0}(\RR^{2n})$ let us assume that it is so for the families $\{|\la|^{N}\d_{\la}^{N}b_{\la}\}_{\la}$, where $N<M$.
Now using (\ref{pelny}) we can write
$$|\la|^{M}\d_{\la}^{M}b_{\la}=\sum_{M\in\Xi}c_{M}|\la|^{M_{1}}\d_{\la}^{M_{1}}b_{\la}\#|\la|^{M_{2}}\d_{\la}^{M_{2}}a_{\la}\#|\la|^{M_{3}}\d_{\la}^{M_{3}}b_{\la}.$$
As $M_{2}>0$ one can use an induction argument and the standard symbolic calculus to deduce that the right hand side is bounded in $Sym^{0}(\RR^{2n})$. Therefore, the families $\{|\la|^{M}\d_{\la}^{M}b_{\la}\}_{\la}$ are bounded, for all $M\in\mathbb{N}$.
Thus, $$|\d_{w}^{\alpha}(|\la|^{M}\d_{\la}^{\beta}b_{\la}(w))|\leqslant c_{\alpha, M}(1+\Vert w\Vert)^{-|\alpha|},$$
so
$$|\d_{w}^{\alpha}\d_{\la}^{\beta}b_{\la}(w)|\leqslant c_{\alpha, M}(1+\Vert w\Vert)^{-|\alpha|}|\la|^{-M}.$$
One can calculate that
$$\d_{w}^{\alpha}\d_{\la}^{M}b_{\la}(w)=\sum_{\gamma+\delta=\alpha}\sum_{1\leqslant |\beta|+j\leqslant M}\frac{c_{\gamma,\beta,j}w^{\beta-\gamma}}{|\la|^{M-j-(|\beta|+|\delta|)/2}}(\d_{1,2,...,2n}^{\beta+\delta}\d^{j}_{2n+1}\widehat{B})(|\la|^{1/2}w,\la).$$
The only component of the sum on the right hand side that includes $j=M$ is the one with $\beta=\gamma=0,\ \delta=\alpha$. So, by induction hypothesis, we have
\begin{align*}
&|(\d_{1,2,...,2n}^{\alpha}\d^{M}_{2n+1}\widehat{B})(|\la|^{1/2}w,\la)||\la|^{|\alpha|/2}\\
&\lesssim\sum_{\gamma+\delta=\alpha}\sum_{1\leqslant |\beta|+j\leqslant M,\ \beta\neq 0}\frac{\Vert w\Vert^{|\beta|-|\gamma|}}{|\la|^{M-j-(|\beta|+|\delta|)/2}(\Vert |\la|^{1/2}w\Vert+|\la|^{1/2})^{|\beta|+|\delta|}|\la|^{j}}\\
&+|\d_{w}^{\alpha}\d_{\la}^{M}b_{\la}(w)|\\
&\lesssim\sum_{\gamma+\delta=\alpha}\sum_{1\leqslant |\beta|+j\leqslant M,\ \beta\neq 0}\frac{\Vert w\Vert^{|\beta|-|\gamma|}}{(1+\Vert w\Vert)^{|\beta|-|\gamma|}(1+\Vert w\Vert)^{|\alpha|}|\la|^{M}}+(1+\Vert w\Vert)^{-|\alpha|}|\la|^{-M}\\
&\lesssim (1+\Vert w\Vert)^{-|\alpha|}|\la|^{-M}.
\end{align*}
Thus,
\begin{align*}
|(\d_{1,2,...,2n}^{\alpha}\d^{M}_{2n+1}\widehat{B})(|\la|^{1/2}w,\la)|&\leqslant c_{\alpha,M}(1+\Vert w\Vert)^{-|\alpha|}|\la|^{-M}|\la|^{-|\alpha|/2}\\
&= c_{\alpha,M}(\Vert |\la|^{1/2}w\Vert+|\la|^{1/2})^{-|\alpha|}|\la|^{-M}.
\end{align*}
Again putting $w$ instead of $|\la|^{1/2}w$ we obtain
$$|(\d_{w}^{\alpha}\d^{M}_{\la}\widehat{B})(w,\la)|\leqslant c_{\alpha,M}(\Vert w\Vert+|\la|^{1/2})^{-|\alpha|}|\la|^{-M}.$$
\end{proof}
\begin{proof}[proof of Theorem \ref{nasze}]
The KN symbol of $(\pi_{K}^{\la})^{-1}$ is $b_{\la}(w)=\widehat{B}(|\la|^{1/2}w,\la)$ which is the symbol of $\pi_{\widetilde{B}}^{\la}$ and $\widetilde{B}$ is a flag kernel. So $(\pi_{K}^{\la})^{-1}=\pi_{\widetilde{B}}^{\la}$. Now
$$\pi_{\delta_{0}}^{\la}=Id=(\pi_{K}^{\la})^{-1}\pi_{K}^{\la}=\pi_{K}^{\la}(\pi_{K}^{\la})^{-1}=\pi_{\widetilde{B}}^{\la}\pi_{K}^{\la}=\pi_{K}^{\la}\pi_{\widetilde{B}}^{\la}=\pi_{\widetilde{B}\star K}^{\la}=\pi_{K\star \widetilde{B}}^{\la}.$$
So $K\star \widetilde{B}=\delta_{0}=\widetilde{B}\star K$. Now putting $L=\widetilde{B}$ for $f\in L^{2}(\Hn)$ we achieve $K\star L\star f=L\star K\star f=f$, which is equivalent to $\op(K)\op(L)f=\op(L)\op(K)f=f$ and finally $\op(K)^{-1}=\op(L)$.
\end{proof}


\section*{Acknowledgements}

The author wishes to express his deep gratitude to P.G{\l}owacki and M.Preisner for their helpful advices in preparing the manuscript.


\end{document}